\documentclass[10pt]{article}
\usepackage{latexsym}
\usepackage{amsmath,amsthm, amssymb}
\usepackage{amssymb,amsfonts}
\usepackage{color}
\usepackage{graphicx}
\usepackage{ gensymb }
\usepackage{enumerate}
\usepackage{epsfig}
\usepackage{verbatim}
\usepackage{mathrsfs}
\usepackage{empheq}
\definecolor{purple}{rgb}{0.65, 0, 1}
\definecolor{orange}{rgb}{1,.5,0}

\def\R{\mathbb{R}}

\def\II{{\rm I\kern-0.5exI}}

\def\III{{\rm I\kern-0.5exI\kern-0.5exI}}
\numberwithin{equation}{section}
\newtheorem{theorem}{Theorem}[section]
\newtheorem{remark}[theorem]{Remark}
\newtheorem{lemma}[theorem]{Lemma}
\newtheorem{definition}[theorem]{Definition}
\newtheorem{proposition}[theorem]{Proposition}
\newtheorem{corollary}[theorem]{Corollary}

\begin{document}
\title{Asymptotic Behavior  for Critical Patlak-Keller-Segel model and an Repulsive-Attractive Aggregation Equation}
\author{Yao Yao\thanks{Department of Mathematics, UCLA. yaoyao@math.ucla.edu}}
\date{}
\maketitle

\begin{abstract}
In this paper we study the long time asymptotic behavior for a class of diffusion-aggregation equations.  Most results except the ones in Section 3.3 concern radial solutions. The main tools used in the paper are maximum-principle type arguments on mass concentration of solutions, as well as energy method.  For the Patlak-Keller-Segel problem with critical power $m=2-2/d$, we prove that all radial solutions with critical mass would converge to a family of stationary solutions, while all radial solutions with subcritical mass converge to a self-similar dissipating solution algebraically fast.  For non-radial solutions, we obtain convergence towards the self-similar dissipating solution when the mass is sufficiently small.  We also apply the mass comparison method to another aggregation model with repulsive-attractive interaction, and prove that radial solutions converge to the stationary solution exponentially fast.
\end{abstract}
\section{Introduction}
Recently there has been a growing interest in the study of nonlocal aggregation phenomena.  The most widely studied models are the Patlak-Keller-Segel  (PKS) models, which describes the cell movement driven by chemotaxis \cite{ks, p, h}.  In this paper, we study the PKS equation in dimension $d\geq 3$ with degenerate diffusion, given by
\begin{equation}\left\{
\begin{split}
 u_t~&= \Delta u^m + \nabla \cdot (u\nabla c)  &x\in\mathbb{R}^d, t\geq 0,\\[0.05cm]
\Delta c~& = u&x\in\mathbb{R}^d, t\geq 0, \\[0.05cm]
~u(x,0)~& = u_0(x) \in L^1_+(\R^d;(1+|x|^2) dx) \cap L^\infty(\R^d)&x\in\mathbb{R}^d,
\end{split}\right. \label{pks}
\end{equation}
here $m>1$ models the local repulsion of cells with anti-crowding effects \cite{tb, tbl}.   This model is a generalization of the classical parabolic-elliptic PKS model in 2D, which has been extensively studied over the years (see the review \cite{h, b} and \cite{dp, bdp, bcm}).
%
%

Throughout this paper, we focus on the critical power $m=m_d:= 2-2/d$, which produces an exact balance  between the diffusion term and the aggregation term  when one performs a mass-invariant scaling.  To study the well-posedness of \eqref{pks}, the following \emph{free energy} functional \eqref{energy} is an important quantity, where the first term is usually referred to as the \emph{entropy} and the latter term is referred to as the \emph{interaction energy}: 
\begin{equation}
\mathcal{F}(u) = \int_{ \mathbb{R}^d} \left(\frac{1}{m-1}u^m +\frac{1}{2}u(u*\mathcal{N}) \right)dx. \label{energy}
\end{equation}
Here $\mathcal{N}(x) = -\frac{1}{(d-2)\sigma_d} |x|^{2-d}$ is the Newtonian potential in $d\geq 3$, with $\sigma_d$ the surface area of the sphere $\mathbb{S}^{d-1}$ in $\mathbb{R}^d$.  The free energy of a weak solution to \eqref{pks} is non-increasing in time; indeed, it is shown that \eqref{pks} is the gradient flow for  $\mathcal{F}$ with respect to the Wasserstein metric (see for example \cite{ags} and \cite{cmv}).

The key observation in \cite{bcl} is the sharp Hardy-Littlewood-Sobolev inequality, which bounds the interaction term in the free energy functional by the entropy term:
\begin{equation}
\left|\int_{\mathbb{R}^d} u(u*\mathcal{N}) dx\right| \leq C^* \|u\|^{2/d}_{L^1(\mathbb{R}^d)} \|u\|^m_{L^m(\mathbb{R}^d)}, \label{vhls}
\end{equation}
where $C^*$ is a constant only depending on the dimension $d$. 
Making use of this inequality, it is proved in \cite{bcl} (and generalized by \cite{brb} for more general kernels) that there exists a critical mass $M_c$ only depending on $d$, which sharply divides the possibility of finite time blow up and global existence.   Although the global existence/blow-up results are well understood, the asymptotic behavior of the solution has not been fully investigated: this motivates our study. In this paper we study the asymptotic behavior of solutions with mass $M\leq M_c$, and try to answer the open questions raised in \cite{bcl} and \cite{b}. 

$\circ$~\textit{Critical case:}

When $M=M_c$, it is proved in \cite{bcl} that the global minimizers of the free energy functional $\mathcal{F}$ have zero free energy, and are given by the one-parameter family
\begin{equation}
u_R(x) = \frac{1}{R^d} u_1(\dfrac{x}{R})
\label{family_stat_solution}
\end{equation}
subject to translations. Here $u_1$ is the unique radial classical solution to
\begin{equation}
\frac{m}{m-1} \Delta u_1^{m-1} + u_1 = 0  \text{ in }B(0,1),  ~\text{ with }u_1 = 0 \text{ on }\partial B(0,1).
\label{eq:u1}
\end{equation}
It was unknown that whether this family of stationary solutions attract some solutions.  

In Section 3.1, we prove the convergence of radial solutions towards this family of stationary solution (see Theorem \ref{thm:m=mc}) using a combination of mass comparison and energy method.  The mass comparison property is a version of comparison principle on the mass distribution of solutions (see Proposition \ref{comp_concentration}), which has been introduced in \cite{ky}. Although it only works for radial solutions, it provides more delicate control than the energy method: it has been recently used in \cite{bk} to prove finite time blow-up of solutions with supercritical mass for a diffusion-aggregation model where Virial identity does not apply.  We mention that the mass comparison property have been previously observed for the porous medium equation \cite{v} and PKS models with linear diffusion \cite{jl, bkln}.

$\circ$~\textit{Subcritical case:}

When $0<M<M_c$, the weak solution exists globally in time, as long as its initial $L^m$-norm is finite \cite{bcl, brb}.  Moreover it has been proved in \cite{bcl} that there exists a dissipating self-similar solution, with the same scaling as the porous medium equation.  However it was unknown whether this self-similar solution would attract all solutions in the intermediate asymptotics.

In Section 3.2,  we prove algebraic convergence of radial solutions towards this self-similar dissipating solution, given that the initial data is bounded and compactly supported.  This is done by constructing explicit barriers in the mass comparison sense. 

For general (non-radial) solutions, the asymptotic behavior of solution to \eqref{pks} is unknown for all mass sizes when $m=2-2/d$. In Section 3.3, we prove that when the mass is sufficiently small, every solution to \eqref{pks} with compactly supported initial data converges to the self-similar dissipating solution. We mention that similar results for small-mass solutions are obtained by \cite{bdef} for the PKS equation in 2D with linear diffusion, however their argument is based on a spectral gap method and cannot be generalized to \eqref{pks} due to the nonlinear diffusion term.

When dealing with non-radial solutions  with small mass,  our key result is the uniqueness of stationary solution (in rescaled variables), which is proved using a maximum principle type argument (see Theorem \ref{thm:unique}). We point out that although it is well known that the global minimizer to the rescaled energy \eqref{rescaled_energy} is unique \cite{bcl}, there are few results concerning the uniqueness of stationary solutions for nonlocal PDEs, except in the following special cases: the stationary solution to a similar equation is proved to be unique by \cite{bdf}  in the 1D case, and stationary solution to the 2D Navier-Stokes equation (in rescaled variable) is proved to be unique by \cite{ggl}, where their proof is based on the fact that all stationary solutions have the same second moment, thus cannot be applied to \eqref{pks}. 



\vspace{0.4cm}

In Section 4 we generalize the mass comparison methods to an aggregation equation 
\begin{equation}
u_t - \nabla \cdot (u \nabla K* u) = 0, \label{rep_att1}
\end{equation}
and prove that the radial solution converges towards the stationary solution exponentially fast for a class of kernel $K$.  Equation \eqref{rep_att1} appears in various contexts as a mathematical model for biological aggregations \cite{me, tbl}.  In order to capture the biologically relevant features of solutions, it is desirable for the interaction kernel $K$ to be repulsive for short-range interactions and attractive for long-range. Aggregation equations with such kinds of  kernel are studied in \cite{bclr, ltb, fhk, fh, ksub}. In this paper we adopt the kernel $K$ proposed in \cite{fhk}, which has a repulsion component in the form of the Newtonian potential $\mathcal{N}$ and an attraction component satisfying the power law:
\begin{equation}
K(x) = \mathcal{N}(x) + \frac{1}{q}|x|^q, \label{def:K1}
\end{equation}
here when $q=0$ the second term is replaced by $\ln|x|$.  We assume that the attraction part is less singular than the repulsion part at the origin, i.e. $q > 2-d$.  In addition, we assume $q\leq 2$, i.e. the long-range attraction does not grow more than linearly as the distance goes to infinity. Note that  in \cite{fh} they mostly focus on the case $q\geq 2$.

The existence and uniqueness of weak solution is established in \cite{fh} for $q>2-d$. In addition, they proved that for any mass there is a unique radial stationary solution which is continuous in its support.  While numerical results in \cite{fh} suggests that this stationary solution should be a global attractor, there is no rigorous proof.   We prove that when $2-d<q\leq 2$, the mass comparison argument can be easily generalized to \eqref{rep_att1}, then we construct barriers in the mass comparison sense to prove that all radial solution converges to the stationary solution exponentially fast.

\emph{Outline of the Paper.}  In Section 2 we state a mass comparison result for a general diffusion-aggregation equation. In section 3 we apply it to the PKS model with critical power, and obtain some asymptotic results for radial solutions with critical and subcritical mass sizes respectively. In Section 4 we demonstrate that the mass comparison can also be applied to an aggregation model with repulsive-attractive interaction, and prove the asymptotic convergence of solution towards the stationary solution.

\subsection{Summary of Results}

By  constructing explicit barriers in the mass comparison sense and using energy method, we obtain the following results for radial solutions of \eqref{pks}:
 
 \begin{theorem}
Suppose $d\geq 3$ and $m=2-2/d$. Let $u(x,t)$ be the weak solution to \eqref{pks} with mass $A$ and initial data $u_0 \in L^1_+(\mathbb{R}^d; (1+|x|^2) dx) \cap L^\infty(\mathbb{R}^d)$, where $u_0$ is continuous, radially symmetric and compactly supported.  Then the following results hold:
\begin{itemize}
\item If $0<A<M_c$, $u(\cdot, t)$ converges to the dissipating self-similar solution $u_A$ as $t\to \infty$, where the Wasserstein distance between $u(\cdot, t)$ and $u_A$ decays algebraically fast as $t\to \infty$. (Corollary \ref{thm:conv_m<mc})
\item If $A=M_c$ and $u_0$ satisfies $\nabla u_0^m \in L^2(\mathbb{R}^d)$ in addition to the assumptions above, then $u(\cdot, t) \to u_{R_0}$ in $L^\infty(\mathbb{R}^d)$ as $t \to \infty$ for some $R_0>0$, where $u_{R_0}$ is a stationary solution defined in \eqref{family_stat_solution}. (Theorem \ref{thm:m=mc})
\end{itemize}
 \end{theorem}

For general (possibly non-radial) initial data, we use a maximum principle type method to prove that when the mass is sufficiently small, every compactly supported stationary solution must be radially symmetric.  This leads to the following asymptotic convergence result:

 \begin{theorem}
Suppose $d\geq 3$ and $m=2-2/d$. Let $u(x,t)$ be the weak solution to \eqref{pks} with mass $0<A<M_c/2$ being sufficiently small, and the initial data $u_0 \in L^1_+(\mathbb{R}^d; (1+|x|^2) dx) \cap L^\infty(\mathbb{R}^d)$ is continuous and compactly supported.  Then we have
$$\lim_{t\to\infty}\|u(\cdot, t) - u_A\|_p=0 \text{ for all }1\leq p \leq \infty,$$
where $u_A$ is the self-similar dissipating solution defined in \eqref{dissipating_self_similar}. (Corollary \ref{cor:nonradial})
\end{theorem}

In Section 4 we generalize the mass comparison methods to the repulsive-attractive aggregation equation \eqref{rep_att}, and obtain the following asymptotic convergence result for $2-d<q\leq 2$:

\begin{theorem}\label{thm:aggregation_intro}
Assume $2-d < q \leq 2$.  Let $u$ be a weak solution to \eqref{rep_att1} with initial data $u_0$ and mass $A$, where $u_0 \in L^1(\mathbb{R}^d) \cap L^\infty(\mathbb{R}^d) $ is non-negative, radially symmetric and compactly supported.  In addition, we assume that $u_0$ is strictly positive in a neighborhood of $0$. Let $u_s$ be the unique stationary solution with mass $A$, as given by Proposition \ref{prop:stat_sol}. Then as $t\to \infty$, $u(\cdot, t)$ converges to $u_s$ exponentially fast in Wasserstein distance. (Theorem \ref{thm:aggregation})
\end{theorem}


\textbf{Acknowledgments:} 
The author is deeply grateful to her advisor Inwon Kim for her guidance, support and suggestions concerning this work.  The author also would like to thank Jacob Bedrossian for helpful discussion and comments.

\section{Mass Comparison for Radial Solutions}
In this section we consider the following type of diffusion-aggregation equation
\begin{equation}\label{general_pde}
u_t = c_1 \Delta u^m + \nabla \cdot (u \nabla (u*(c_2 \mathcal{N} + c_3 \mathcal{K}) +  V)),
\end{equation}
where $\mathcal{N}(x) = -\frac{1}{(d-2)\sigma_d} |x|^{2-d}$ is the Newtonian potential in $d\geq 3$, with $\sigma_d$ the surface area of the sphere $\mathbb{S}^{d-1}$ in $\mathbb{R}^d$.
We make the following assumptions on the kernel $\mathcal{K}$, the potential $V$ and the coefficients:
\begin{itemize}
\item[\textbf{(C)}]$c_1, c_3 \geq 0$, and $c_2 \in \mathbb{R}$ can be of any sign.
\item[\textbf{(K1)}]$\mathcal{K}$ is radially symmetric.
\item[\textbf{(K2)}]$\Delta \mathcal{K} \in L^1(\mathbb{R}^d)$, $\Delta \mathcal{K} \geq 0$ and is radially decreasing.
\item[\textbf{(V1)}]$V \in C^2(\mathbb{R}^d)$ is radially symmetric.

\end{itemize}

For a radially symmetric function $u(x,t)$, we define its mass function $M(r,t; u)$ by
\begin{equation}
M(r,t; u) := \int_{B(0,r)} u(x,t) dt,
\label{def_M}
\end{equation}
and we may write it as $M(r,t)$ if the dependence on the function $u$ is clear.  The following lemma describes the PDE satisfied by the mass function.

\begin{lemma}[Evolution of Mass Function]
Let $u(x,t)$ be a non-negative smooth radially symmetric solution to \eqref{general_pde}. Let $M(r,t) = M(r,t;u)$ be as defined in \eqref{def_M}.  Then $M(r,t)$ satisfies
\begin{equation}
\frac{\partial M}{\partial t} = c_1 \sigma_d r^{d-1} {\partial_r} (\frac{\partial_r M}{\sigma_d r^{d-1}})^m + \partial_r M \frac{c_2 M + c_3 \tilde M}{\sigma_d r^{d-1}} + \partial_r M \partial_r V, \label{pde_for_m}
\end{equation}
where $\tilde M(r,t; u) := \int_{B(0,r)} u*\Delta \mathcal{K} dx$.

\end{lemma}

\begin{proof}
Due to divergence theorem and radial symmetry of $u$, we have
\begin{equation}
\frac{\partial M}{\partial t} = \sigma_d r^{d-1} [ c_1\partial_r u^m + u(\partial_r(u*(c_2 \mathcal{N} + c_3 \mathcal{K}))+ \partial_r V)]. \label{eq:dm/dt}
\end{equation}
Note that radial symmetry of $u$ also gives 
\begin{equation}
u(r) = \frac{\partial_r M }{\sigma_d r^{d-1}}. \label{div0}
\end{equation} It remains to write $\partial_r (u*\mathcal{N})$ and $\partial_r (u*\mathcal{K})$ in terms of $M$.  For $\partial_r (u*\mathcal{N})$, divergence theorem gives
\begin{equation}
\partial_r (u*\mathcal{N}) = \frac{\int_{B(0,r)} \Delta u*\mathcal{N} dx}{\sigma_d r^{d-1}} = \frac{M(r,t)}{\sigma_d r^{d-1}}. \label{div1}
\end{equation}

We can similarly obtain
\begin{equation}
\partial_r (u*\mathcal{K}) = \frac{\int_{B(0,r)} u*\Delta \mathcal{K} dx}{\sigma_d r^{d-1}} = \frac{\tilde M(r,t)}{\sigma_d r^{d-1}}, \label{div2}
\end{equation}
where $\tilde M(r,t; u) := \int_{B(0,r)} u*\Delta \mathcal{K} dx$.  Plug \eqref{div0}, \eqref{div1} and \eqref{div2} into equation \eqref{eq:dm/dt}, and then we can obtain \eqref{pde_for_m}.
\end{proof}

\begin{definition}
Let $u_1$ and $u_2$ be two non-negative radially symmetric functions in $L^1(\mathbb{R}^d)$.  We say $u_1$ is \emph{less concentrated than} $u_2$, or $u_1 \prec u_2$, if
$$\int_{B(0,r)} u_1(x) dx \leq \int_{B(0,r)} u_2(x)dx\quad\hbox{ for all } r\geq 0.$$
\end{definition}

\begin{definition}
Let $u_1(x,t)$ be a non-negative, radially symmetric function in $L^1(\mathbb{R}^d)\cap L^\infty(\mathbb{R}^d)$, which is $C^1$ in its positive set. We say $u_1$ is a \emph{supersolution of  \eqref{general_pde} in the mass comparison sense} if $M_1(r,t):= M(r,t;u_1)$ is a supersolution of \eqref{pde_for_m}, i.e. $M_1(r,t)$ and $\tilde{M}_1(r,t):= \tilde{M}(r,t;u_1)$ satisfy
\begin{equation} \label{ineq_m1}
\frac{\partial M_1}{\partial t} \geq c_1 \sigma_d r^{d-1} {\partial_r} (\frac{\partial_r M_1}{\sigma_d r^{d-1}})^m + \partial_r M_1\frac{c_2 M_1 + c_3 \tilde M_1}{\sigma_d r^{d-1}} + \partial_r M_1 \partial_r V,
\end{equation}
in the positive set of $u_1$.

Similarly we can define a \emph{subsolution} of \eqref{general_pde} in the mass comparison sense.
\end{definition}

\begin{proposition}[\textbf{mass comparison}] Suppose $m>1$, and $c_1, c_2, c_3, V, \mathcal{K}$ satisfy the assumptions \textbf{(C),(K1),(K2),(V1)}. Let $u_1(x,t)$  be a supersolution and $u_2(x,t)$ be a subsolution of \eqref{general_pde} in the mass comparison sense for $t\in[0,T]$. Further assume that both $u_i$ are bounded, and $u_i$'s  preserve their mass over time, i.e.,
$\int u_1(\cdot, t) dx$ and $\int u_2(\cdot, t) dx$ stay constant for all $0\leq t \leq T.$  Then their mass functions are ordered for all times: i.e.,
if  $u_1(x,0) \succ u_2(x,0)$, then we have $u_1(x,t) \succ u_2(x,t)$ for all $t\in[0,T]$.
\label{comp_concentration}
\end{proposition}

\begin{proof}Let $M_i(r,t)$ be the mass function for $u_i$, where $i=1,2$. We claim that $M_1(r,t) \geq M_2(r,t)$ for all $r\geq0$ and $t\in [0,T]$, which proves the proposition.

For the boundary conditions of $M_i$, note that
\begin{equation*}
\begin{cases}
M_1(0,t)=M_2(0,t)=0 \quad \text{for all $t\in[0,T]$},\\
\lim_{r\to\infty}(M_1(r,t)-M_2(r,t)) = \int_{\mathbb{R}^d} (u_1(x,0) - u_2(x,0)) dx\geq 0  \quad \text{for all $t\in[0,T]$}.
\end{cases}
\end{equation*}
As for initial data, we have $M_1(r,0)\geq M_2(r,0)$ for all $r\geq 0$.  

For given $\lambda>0$, we define
$$w(r,t):=\big(M_2(r,t)-M_1(r,t)\big)e^{-\lambda t},$$
where $\lambda$ is a large constant to be determined later. Suppose the claim is false, then $w$ attains a positive maximum at some point $(r_1, t_1)$ in the domain $(0,\infty)\times(0,T
]$.
Moreover, since the mass of both $u_1$ and $u_2$ are preserved over time and thus are ordered, we know that $(r_1,t_1)$ must lie inside the positive set for both $u_1$ and $u_2$, where $M_i$'s are $C^{2,1}_{x,t}$.

Since $w$ attains a maximum at $(r_1, t_1)$, the following inequalities hold at $(r_1, t_1)$:
\begin{eqnarray}
&&w_t \geq 0~\Longrightarrow \partial_t (M_2-M_1) \geq \lambda (M_2-M_1)\label{compare_t}\\[0.1cm]
&&w_r=0~\Longrightarrow\partial_r M_1 = \partial_r M_2  > 0\label{compare_r}\\[0.1cm]
&&w_{rr}\leq 0 ~\Longrightarrow  \partial_{rr} M_1 \geq \partial_{rr} M_2\label{compare_rr}
\end{eqnarray}
Now we will analyze the terms on the right hand side of \eqref{ineq_m1} one by one. For the first term, \eqref{compare_r} and \eqref{compare_rr} imply that
\begin{equation}\label{ineq:pme}
 c_1 {\partial_r} (\frac{\partial_r M_2}{\sigma_d r^{d-1}})^m - c_1 {\partial_r} (\frac{\partial_r M_1}{\sigma_d r^{d-1}})^m \leq 0 ~~~~\text{at } (r_1, t_1).
\end{equation}

For the term coming from Newtonian potential, we have
\begin{equation}\label{ineq:newtonian}
\partial_r M_1\frac{c_2 (M_2-M_1)}{\sigma_d r^{d-1}}= c_2 u_1(r_1, t_1) (M_2-M_1)\leq u_{\max} |c_2| (M_2-M_1),\end{equation}
where $u_{\max} := \max\{\sup_{\mathbb{R}^d\times[0,T]} u_1, \sup_{\mathbb{R}^d\times[0,T]} u_2\}$ is finite by assumption on $u_1$ and $u_2$.

 We next claim 
 \begin{equation}
 \partial_r M_1 \frac{c_3(\tilde M_2 - \tilde M_1)(r_1, t_1)}{\sigma_d r^{d-1}} \leq u_{\max} c_3 \|\Delta \mathcal{K}\|_1(M_2-M_1)(r_1, t_1).\label{ineq:K}
 \end{equation}
To prove the claim, note that $\tilde M_2 - \tilde M_1$ can be rewritten as 
  \begin{eqnarray}
\nonumber\tilde M_2(r_1, t_1) - \tilde M_1(r_1, t_1) &=& \int_{\mathbb{R}^d} ((u_2-u_1)*\Delta \mathcal{K}) ~\chi_{B(0,r_1)} dx\\
&=& \int_{\mathbb{R}^d} (u_2-u_1)~ (\chi_{B(0,r_1)} * \Delta \mathcal{K}) dx. \label{eq:m_tilde}
 \end{eqnarray}
Note that $\Delta \mathcal{K}\geq 0$ is radially decreasing due to assumption \textbf{(K2)}, thus $\chi_{B(0,r_1)} * \Delta \mathcal{K}$ is non-negative, radially decreasing and has maximum less than or equal to $\|\Delta \mathcal{K}\|_{1}$.  Therefore we can use a sum of bump function to approximate $\chi_{B(0,r_1)} * \Delta \mathcal{K}$, where the sum of the height is less than $\|\Delta \mathcal{K}\|_1$.
Hence
 $$
 \tilde M_2(r_1, t_1) - \tilde M_1(r_1, t_1)  \leq \|\Delta \mathcal{K}\|_{1} \sup_x (M_2 - M_1)(x, t_1) = \|\Delta \mathcal{K}\|_1 (M_2-M_1)(r_1,t_1),
 $$
 which proves the claim \eqref{ineq:K}. 
 Finally, for the last term in \eqref{ineq_m1}, as a result of \eqref{compare_r} we have 
 \begin{equation}
 \partial_r M_2 \partial V - \partial_r M_1 \partial V = 0.
 \label{ineq:V}
 \end{equation}
 Now we subtract \eqref{ineq_m1} with the corresponding equation for the subsolution.  Due to the inequalities \eqref{ineq:pme}, \eqref{ineq:newtonian}, \eqref{ineq:K} and \eqref{ineq:V}, we obtain that
$$\partial_t (M_2 - M_1) \leq u_{\max} (|c_2| + c_3\|\Delta K\|_1) (M_2 - M_1).$$
 Hence if we choose $\lambda > u_{\max} (|c_2| + c_3 \|\Delta K\|_1) $ in the beginning of the proof, the inequality above will contradict \eqref{compare_t}.
\end{proof}
\begin{remark}
If both $u_1$ and $u_2$ are supported in some compact set $B(0,R)$ for $0\leq t \leq T$, then the assumption \textbf{(K2)} can be replaced by \textbf{(K2')} as follows:

\textbf{\emph{(K2')}}~~$\Delta K \in L^1_{loc}(\mathbb{R}^d)$, $\Delta K \geq 0$ and is radially decreasing.

The proof under condition \textbf{(K2')} is almost the same, except that there is some change in \eqref{ineq:K}. Note that in this case \eqref{eq:m_tilde} still holds, and we can bound the maximum of $\chi_{B(0,r_1)} * \Delta \mathcal{K}$ by $\int_{B(0,R)} \Delta \mathcal{K} dx$.  This yields an inequality similar to \eqref{ineq:K}, with the $\|\Delta \mathcal{K}\|_1$ in the right hand side replaced by $\int_{B(0,R)} \Delta \mathcal{K} dx$. \label{remark:l1loc}
\end{remark}

\section{Application to Critical Patlak-Keller-Segel Models}

\subsection{Convergence towards stationary solution for critical mass}
In this subsection, we prove that every radial solution with mass $M_c$ and continuous, compactly supported initial data will be eventually attracted to some stationary solution within the family \eqref{family_stat_solution}.

If the initial data is bounded above and has the critical mass $M_c$, then it is proved in \cite{bcl} that the weak solution to \eqref{pks} exists globally in time. In the next lemma we prove the solution has a global (in time) $L^\infty$ bound. In addition, if the initial data is  radially symmetric and compactly supported, then the support of the solution would stay uniformly bounded in time.

\begin{lemma}
Suppose $d \geq 3$ and $m=2-2/d$. Consider the problem \eqref{pks} with initial data $u_0 \in L^1_+(\mathbb{R}^d; (1+|x|^2) dx) \cap L^\infty(\mathbb{R}^d)$, where $u_0$ is continuous and has critical mass $M_c$. Then the $L^\infty$ norm of the weak solution $u(x,t)$ is globally bounded, i.e. there exists $K>0$ depending on $\|u_0\|_\infty$ and $d$, such that $\|u(\cdot, t)\|_{L^\infty(\mathbb{R}^d)} \leq K_1$ for all $t\geq 0$.

If $u_0$ is radially symmetric and compactly supported in addition to the assumptions above, then there exists some $R_2>0$, such that $\{u(\cdot, t)>0\} \subseteq B(0,R_2)$ for all $t\geq0$, where $R_2$ depend on $d$ and $u_0$.
\label{lemma:u_bounded}
\end{lemma}

\begin{proof}
In order to bound the $L^\infty$ norm of $u(\cdot, t)$, we first consider equation \eqref{pks} with symmtetrized initial data, which is described below.  Let $\bar u(\cdot, t)$ be the solution to \eqref{pks} with initial data $u_0^*$, where $ u_0^*(\cdot)$ is the radial decreasing rearrangement of $u_0$. Here the \emph{radial decreasing rearrangement} of a non-negative function $f$ is defined as
\begin{equation}\label{def:rearrangement}
f^*(x):=\int_0^\infty \chi_{\{f>t\}^*}(x) dt.
\end{equation}
Since $\bar u$ has a radially symmetric initial data and has mass $M_c$, due to \cite{bcl}, we readily obtain that $\bar u$ exists globally in time, and $\bar u$ is radially symmetric for all $t\geq 0$. We first prove that there is a global $L^\infty$ bound for $\bar u$. 

Since $\|\bar u(\cdot, 0)\|_\infty = \|u_0\|_\infty < \infty$, we can choose $R_1>0$ depending on $\|u_0\|_\infty$, where $R_1$ is sufficiently small such that $u_0^* \prec u_{R_1}$, where $u_{R_1}$ is as defined in \eqref{family_stat_solution}. Then the mass comparison result in Proposition \ref{comp_concentration} yields that  
\begin{equation}
\bar u(\cdot, t) \prec u_{R_1}\text{ for all }t\geq 0.
\end{equation}
Now we go back to the original solution $u$, and compare $u$ with $\bar u$. 
It is proved in Theorem 6.3 of \cite{ky} that 
$$u^*(\cdot, t) \prec \bar u(\cdot, t) \text{ for all }t\geq 0.$$
Combining the above two inequalities together, we readily obtain that
$$u^*(\cdot, t) \prec u_{R_1}  \text{ for all }t\geq 0,$$
which implies that $u^*(0, t) \leq u_{R_1}(0) \text{ for }t\geq 0$.
Note that $u^*(\cdot, t)$ is radially decreasing for all $t\geq 0$ by definition, and $u_{R_1}$ is radially decreasing due to \cite{bcl}. Hence the above inequality implies that 
\begin{equation}
\|u(\cdot, t)\|_\infty = \|u^*(\cdot, t)\|_\infty \leq \| u_{R_1}\|_\infty = R_1^{-d} u_1(0) \text{~ for }t\geq 0,
\end{equation}
thus $u$ has a global $L^\infty$ bound $R_1^{-d} u_1(0)$, where $u_1$ is as defined in \eqref{family_stat_solution}.

Next we hope to show that if $u_0$ is  radially symmetric and compactly supported in addition to the conditions above, the support of $u(\cdot, t)$ will stay in some compact set for all time.  We first prove it for the case where $u_0(0)>0$. Due to the continuity of $u_0$, we have $u_0$ is uniformly positive in a neighborhood of $0$. This enables us to choose $R_2>0$ sufficiently large such that $u_0 \succ u_{R_2} $, where $u_{R_2}$ is as defined in \eqref{family_stat_solution}. Proposition \ref{comp_concentration} again gives us
$u(\cdot, t) \succ u_{R_2} \text{ for all }t\geq 0,$
which implies that $$\text{supp}~u(\cdot, t) \subseteq \text{supp}~u_{R_2} = B(0, R_2)  \text{ for all }t\geq 0.$$

If $u_0(0)=0$, we claim that after some finite time $t_1$, $u(0,t_1)$ becomes positive, and $u(\cdot, t_1)$ has a compact support, where $t_1$ depends on $d$ and $u_0$. Then we can take $t_1$ as the starting time and argue as in the case $u_0(0)>0$.

Now we will prove the claim. This is done by performing the mass comparison between $u$ and $w$, where $w$ is the solution to the porous medium equation
\begin{equation}
w_t = \Delta w^m \text{ in }\mathbb{R}^d \times [0,\infty), \label{pme_w}
\end{equation}
with initial data $w(\cdot, 0) = u_0$. It can be readily checked that $u$ is a supersolution of \eqref{pme_w} in the mass comparison sense, hence Proposition \ref{comp_concentration} implies that $u(\cdot, t) \succ w(\cdot, t)$ for all $t\geq 0$, which yields 
\begin{equation}u(0,t) \geq w(0,t) \text{ ~and ~}\text{supp} ~u(\cdot, t) \subseteq \text{supp} ~w(\cdot, t) \text{ ~for all }t\geq 0.
\label{compare_u_w}
\end{equation} 
For  the porous medium equation \eqref{pme_w}, it is well known that the solution will eventually converge to the self-similar Barenblatt profile (see \cite{v} for example), which has a positive density at $0$ for all $t\geq 0$. Hence there exists some $t_1 \geq 0$ such that $w(0,t_1)>0$. In addition, $w(\cdot,t_1)$ has a compact support, due to the finite speed of propagation property of porous medium equation \cite{v}.  Therefore \eqref{compare_u_w} yields that  $u(0,t_1)\geq  w(0,t_1) > 0$ and $u(\cdot,t_1)$ is compactly supported, which prove the claim.
 \end{proof}

Next we prove that under the conditions in Lemma \ref{lemma:u_bounded}, every radial solution converges to some stationary solution in the family \eqref{family_stat_solution} as $t\to \infty$. To do this we investigate the free energy functional \eqref{energy}, and make use of the following result proved in \cite{bcl} and \cite{brb}: Let $u$ be a weak solution to \eqref{pks}, then it satisfies the following energy dissipation inequality for almost every $t$ during its existence:
\begin{equation}
\mathcal{F}(u(t)) + \int_0^t \int_{\mathbb{R}^d} u \left |\frac{m}{m-1} \nabla u^{m-1} + \nabla \mathcal{N}*u\right |^2 dxdt \leq \mathcal{F}(u_0).
\label{energy_ineq}
\end{equation}

\begin{theorem}
Suppose $d \geq 3$ and $m=2-2/d$.  Let $u(x,t)$ be the weak solution to \eqref{pks} with critical mass $M_c$ and  initial data $u_0 \in L^1_+(\mathbb{R}^d; (1+|x|^2) dx) \cap L^\infty(\mathbb{R}^d)$, where $u_0$ is continuous, radially symmetric and compactly supported, and satisfies $\nabla u_0^m \in L^2(\mathbb{R}^d)$. Then  there exists $R_0 > 0$ depending on $u_0$ and $d$, such that $u(\cdot, t) \to u_{R_0}$ in $L^\infty(\mathbb{R}^d)$ as $t \to \infty$,  where $u_{R_0}$ is as defined in \eqref{family_stat_solution}.
\label{thm:m=mc}
\end{theorem}

\begin{proof}
Due to Lemma \ref{lemma:u_bounded}, we obtain the existence of a weak solution globally in time, which has a global $L^\infty$ bound.  In addition, by treating $u*\mathcal{N}$ as an \emph{a priori} potential in \eqref{pks} and applying the continuity result in \cite{dib}, we obtain that $u(x,t)$ is uniformly continuous in space and time in $[\tau, \infty)$ for all $\tau>0$.   

Our preliminary goal is to find a time sequence $\{t_n\}$ which increases to infinity, such that $u(t_n)$ uniformly converges to some stationary solution as $n\to \infty$.  Note that $\mathcal{F}(u(\cdot, t))$ is non-increasing for almost every $t$ due to \eqref{energy_ineq}, and is bounded below as $t\to\infty$. This enables us to find a time sequence $\{t_n\}$ which increases to infinity, such that
\begin{equation}\label{u_n}
\lim_{n\to\infty} \int_{\mathbb{R}^d} u(t_n) \left |\frac{m}{m-1} \nabla u(t_n)^{m-1} + \nabla \mathcal{N}*u(t_n)\right |^2 dx = 0.
\end{equation}
We will slightly abuse the notation and denote $u(t_n)$ by $u_n$.   Note that $\{u_n\}$ is uniformly bounded and equicontinuous, hence Arzel\`{a}-Ascoli theorem enables us to find a subsequence of $\{u_n\}$, such that
\begin{equation}\label{unif_conv_to_u_infty}
u_n \to u_\infty \quad\text{uniformly in }n, 
\end{equation}
where $u_\infty$ is some radially symmetric and continuous function. Moreover, Lemma \ref{lemma:u_bounded} implies that the support of $\{u_n\}$ all stays in some fixed compact set, hence we have $u_\infty$ is compactly supported as well, and it has mass $M_c$. We will prove that $u_\infty$ is indeed a stationary solution later.

We next claim that $\{\nabla u_n^m\}$ are uniformly bounded in $L^2(\mathbb{R}^d)$.  To prove the claim, note that
\begin{eqnarray*}
\int_{\mathbb{R}^d} |\nabla u^m_n + u_n \nabla \mathcal{N}*u_n|^2 dx \leq \| u_n\|_\infty \int_{\mathbb{R}^d}  u_n \left|\frac{m}{m-1}\nabla u^{m-1}_n + \nabla \mathcal{N}*u_n\right|^2 dx,
\end{eqnarray*}
where the right hand side is uniformly bounded for all $n$.  In addition, since  $\{u_n\}$ are uniformly bounded and are all supported in some $B(0,R)$, we know $\int_{\mathbb{R}^d} u_n |\nabla \mathcal{N}*u_n|^2 dx$ is also uniformly bounded for all $n$.  Therefore triangle inequality yields the uniform boundedness of $\int_{\mathbb{R}^d} |\nabla u^m_n |^2 dx$, which proves the claim.

The uniform boundedness of $\{\nabla u_n^m\}$ in $L^2(\mathbb{R}^d)$ implies that $\{\nabla u_\infty^m\}$ is in $L^2(\mathbb{R}^d)$ as well. Moreover, we can find a subsequence of $\{u_n\}$ (which we again denote by $\{u_n\}$ for the simplicity of notation), such that 

\begin{equation}
\nabla u_n^{m} \rightharpoonup \nabla u_\infty^{m} \text{ as } n\to\infty  \text{ weakly in } L^1(\mathbb{R}^d_L: \mathbb{R}^d). \label{conv_2}
\end{equation}
Using \eqref{u_n} and the two convergence properties \eqref{unif_conv_to_u_infty} and \eqref{conv_2}, we can proceed in the same way as Lemma 10 in \cite{cjm} and prove that $u_\infty$ satisfies
\begin{equation}
\int_{\mathbb{R}^d}  u_\infty \left|\frac{m}{m-1}\nabla u^{m-1}_ \infty + \nabla \mathcal{N}*u_ \infty \right|^2 dx = 0, \label{u_is_stat}
 \end{equation}
which implies that $u_\infty$ is a radial stationary solution to \eqref{pks}, hence is indeed in the family \eqref{family_stat_solution}.

Next we will prove that $u(\cdot, t) \to u_\infty$ uniformly in $L^\infty(\mathbb{R}^d)$ as $t\to\infty$. In order to prove this, we make use of the monotonicity of the second moment of $u(\cdot, t)$ in time. By combining the following Virial identity
 \begin{equation} \label{virial}
 \frac{d}{dt}\int_{\mathbb{R}^d} |x|^2 u(x,t) dx = 2(d-2)\mathcal{F}[u(t)] \text{ for all }t
 \end{equation}
 with the fact that the minimizer of $\mathcal{F}$ has free energy $0$, it is shown in \cite{bcl} that
\begin{equation}
M_2[u(\cdot, t)] := \int_{\mathbb{R}^d} |x|^2 u(x,t)  dt \text{ is non-decrasing in }t.\label{m2_nondec}
\end{equation}
This implies that any subsequence of $u(\cdot, t)$ can converge to only one limit: if not, then we can find some another sequence $\{t_n' \}$ increasing to infinity, such that $u(t_n')$ converges to another stationary solution $u_\infty'$ uniformly as $n\to \infty$, where $u_\infty'$ is also in the family \eqref{family_stat_solution}. Since $u(t_n')$ are uniformly bounded and uniformly compactly supported, we have $M_2[u(t_n')] \to M_2[u_\infty']$. On the other hand for the time sequence $\{t_n\}$ we have $M_2[u(t_n)] \to M_2[u_\infty]$, hence \eqref{m2_nondec} implies that $u_\infty$ and $u_\infty'$ must have the same second moment. Since both $u_\infty$ and $u_\infty'$ are within the family \eqref{family_stat_solution}, they can have the same second moment only if they are the same stationary solution.
\end{proof}

\begin{remark}
\textup{Since the proof is done by extracting a subsequence of time, we are unable to obtain the rate of the convergence.  Moreover, it is unknown whether the limit $u_\infty$ has continuous dependence on the initial data $u_0$.  We also point out that the above proof is for radial solution only; for general initial data the difficulty lies in the fact that we are unable to bound the solution in some compact set uniform in time.}
\end{remark}

\subsection{Convergence towards self-similar solution for subcritical mass, radial case}

In this subsection, we prove that every radial solution with subcritical mass and compactly supported initial data would converge to some self-similar solution which is dissipating with the same scaling as the solution of the porous medium equation. 

Let $u$ be the weak solution to \eqref{pks}, with mass $A\in (0,M_c)$.  Following \cite{v} and \cite{bcl}, we rescale $u$ according to the scaling from porous medium equation:
\begin{equation}\label{def_mu}
\mu(\lambda, \tau) = (t+1) u(x,t);\quad \lambda = x(t+1)^{-1/d};\quad \tau=\ln (t+1).
\end{equation}
Then $\mu(\lambda, 0) = u(x,0)$, and $\mu(\lambda,\tau)$ solves the following rescaled equation
\begin{equation}
\mu_\tau = \Delta \mu^m + \nabla\cdot(\mu \nabla\frac{|\lambda|^2}{2d}) + \nabla \cdot(\mu \nabla(\mu*\mathcal{N})). \label{after_scaling}
\end{equation}
It is pointed out in Theorem 5.2 of \cite{bcl} that the free energy associated to the rescaled problem \eqref{after_scaling} is 
\begin{equation}\label{rescaled_energy}
\mathcal{G}(\mu(\cdot, t)) := \int_{\mathbb{R}^d} \left( \frac{m}{m-1} \mu^m + \frac{1}{2}\mu(\mathcal{N}*\mu) + \frac{|\lambda|^2 \mu}{2d} \right) d\lambda,
\end{equation}
and for any mass $A\in (0, M_c)$, there is a unique minimizer $\mu_A$ of $\mathcal{G}$ in $\mathcal{Z}_A$ subject to translation, where $\mathcal{Z}_A := \{h \in L^1(\mathbb{R}^d) \cap L^m(\mathbb{R}^d):  \|h\|_1 = M \text{ and } \int_{\mathbb{R}^d} |x|^2 h(x) dx \leq \infty\}$.  In addition, $\mu_A$ is continuous, radially decreasing and has a compact support, and $\mu_A$ satisfies
\begin{equation}
\dfrac{m}{m-1}\frac{\partial}{\partial r} \mu_A^{m-1} + \frac{r}{d} + \dfrac{M(r; \mu_A)}{\sigma_d r^{d-1}} = 0 \label{eq:stationary_mu}
\end{equation}
in its positive set, where the mass function $M$ is as defined in \eqref{def_M}.

Since $\mu_A$ is a stationary solution of \eqref{after_scaling}, if we go back to the original scaling, $\mu_A$ gives a self-similar solution of \eqref{pks}:
\begin{equation}
u_A(x,t) = (t+1)^{-1} \mu_A(\dfrac{x}{(t+1)^{1/d}}), \label{dissipating_self_similar}
\end{equation}
It is then asked in \cite{bcl} and \cite{b} that whether this self-similar solution attracts all global solutions.  

We will first prove that all radial solutions to the rescaled equation \eqref{after_scaling} converge to $\mu_A$.  The following lemma construct a family of explicit solutions to \eqref{after_scaling}, which all converge to $\mu_A$ exponentially fast as $\tau \to \infty$.

\begin{lemma}[\textbf{A family of explicit solutions}]  Suppose $d \geq 3$ and $m=2-2/d$. For $0<A<M_c$, we denote by $\mu_A$ the stationary solution of \eqref{after_scaling}. Let $\bar \mu$ be defined as 
\begin{equation}
\bar\mu (\lambda,\tau) := \dfrac{1}{R^{d}(\tau)} \mu_A(\dfrac{\lambda }{R(\tau)}),\label{def:mu_bar}
\end{equation} 
where $R(\tau)$ solves the ODE
\begin{equation}\left\{
\begin{split}
\dot{R}(\tau) &= \frac{1}{d}(\dfrac{1}{R^d}-1)R\\
R(0) &= R_0,
\end{split}\right. \label{def:R1}
\end{equation}
where $R_0 > 0$ is a constant.  Then for any $R_0 >0$, $\bar \mu(\lambda,\tau)$ is a weak solution to \eqref{after_scaling}.
\label{lemma:family_solution}
\end{lemma}

\begin{proof}
Since $\bar\mu$ is a self-similar function, it can be easily verified that $\bar \mu$ solves the following transport equation
$$\bar \mu_\tau + \nabla \cdot (\bar \mu \frac{\dot R(\tau)}{R(\tau)} \lambda) = 0.$$
On the other hand, note that \eqref{after_scaling} can also be written as a transport equation
$$\mu_\tau = \nabla \cdot (\mu \vec v),$$
where 
$$\vec v = \dfrac{m}{m-1} \nabla \mu^{m-1} + \frac{\lambda}{d} + \dfrac{M(|\lambda|,\tau;\mu)}{\sigma_d |\lambda|^{d-1}}\dfrac{\lambda}{|\lambda|}.$$
Therefore, to prove that $\bar \mu$ solves \eqref{after_scaling}, it suffices to verify that 
\begin{equation}
-\frac{\dot R(\tau)}{R(\tau)} r =  \dfrac{m}{m-1} \underbrace{\dfrac{\partial}{\partial r} \bar\mu^{m-1}}_{T_1} + \frac{r}{d} + \underbrace{\dfrac{M(r,\tau;\bar \mu)}{\sigma_d r^{d-1}}}_{T_2} \text{ ~for }0\leq r \leq R(\tau).\label{bar_mu}
\end{equation}
Since $\bar \mu$ is a rescaling of $\mu_A$,
\begin{eqnarray*}
T_1 &=& \dfrac{1}{R^{(m-1)d+1}} \left(\dfrac{\partial}{\partial r} \mu_A^{m-1}\right) (\frac{r}{R(\tau)}) = \dfrac{1}{R^{d-1}} \left(\dfrac{\partial}{\partial r} \mu_A^{m-1}\right) (\frac{r}{R(\tau)}),
\end{eqnarray*}
where in the last inequality we used the fact that $m$ is the critical power, i.e. $m=2-2/d$. For $T_2$ in \eqref{bar_mu}, the definition of $\bar \mu$ gives
\begin{eqnarray*}
T_2 &=& \dfrac{1}{R(\tau)^{d-1}} \dfrac{M(\frac{r}{R(\tau)}; \mu_A)}{\sigma_d (\frac{r}{R(\tau)})^{d-1}}  \text{ ~for }0\leq r \leq R(\tau).
\end{eqnarray*}
Now recall that $\mu_A$ satisfies \eqref{eq:stationary_mu} in its positive set, which implies
\begin{eqnarray*}
\text{RHS of \eqref{bar_mu}} &=& \dfrac{1}{R(\tau)^{d-1}}\left(  \frac{m}{m-1}\dfrac{\partial}{\partial r}\mu_A^{m-1} (\frac{r}{R(\tau)}) + \dfrac{r}{dR(\tau)} + \dfrac{M(\frac{r}{R(\tau)}; \mu_A)}{\sigma_d (\frac{r}{R(\tau)})^{d-1}}  \right) + \frac{1}{d}(1 - \dfrac{1}{R^{d}(\tau)})r\\
&=& \frac{1}{d}(1 - \dfrac{1}{R^{d}(\tau)})r\\
&=& -\frac{\dot R(\tau)}{R(\tau)} r,
\end{eqnarray*}
where the last equality comes from the definition of $R$ in \eqref{def:R1}. This verifies that \eqref{bar_mu} is indeed true, which completes the proof.
\end{proof}

Next we use the family of explicit solution constructed above as barriers, and perform mass comparison between the real solution and the barriers.

\begin{proposition}
Suppose $d \geq 3$ and $m=2-2/d$.  Let $\mu(\lambda,\tau)$ be a radial weak solution to \eqref{after_scaling} with mass $0<A<M_c$, where the initial data  $\mu(\cdot, 0) \in L^1_+(\mathbb{R}^d; (1+|x|^2) dx) \cap L^\infty(\mathbb{R}^d)$ is continuous and compactly supported. Then as $\tau \to \infty$, the mass function of $\mu$ converges to the mass function of $\mu_A$ exponentially, i.e. 
$$\sup_r |M(r, \tau; \mu) - M(r, \tau; \mu_A)| \leq Ce^{-\tau},$$
where $\mu_A$ is as defined in \eqref{eq:stationary_mu}, and $C$ depends on $d, A$ and $\mu(\cdot, 0)$.
\label{prop:conv_rescale}
\end{proposition}

\begin{proof}
Without loss of generality we assume that $\mu(0,0)>0$. (When $\mu(0,0)=0$, from the same discussion in \ref{lemma:u_bounded}, $\mu(0,\tau)$ will become positive after some finite time.)   Then we can find $R_{01}$ sufficiently small and $R_{02}$ sufficiently large, such that
$$\dfrac{1}{R_{02}^d}\mu_A(\dfrac{\cdot}{R_{02}})  \prec \mu(\cdot, 0) \prec \dfrac{1}{R_{01}^d}\mu_A(\dfrac{\cdot}{R_{01}}),$$
where in the first inequality we used that $\mu(0,0)>0$, and in the second inequality we used $\|\mu(\cdot, 0) \|_\infty < \infty$.

Let $\mu_1(\lambda, \tau)$ and $\mu_2(\lambda, \tau)$ be defined as in \eqref{def:mu_bar}, with $R(0)$ equal to $R_{01}$ and $R_{02}$ respectively.  Then Lemma \ref{lemma:family_solution} says that both $\mu_1$ and $\mu_2$ are solutions to \eqref{after_scaling}.  Note that \eqref{after_scaling} is a special case of  \eqref{general_pde}, hence the mass comparison result in Proposition \ref{comp_concentration} holds here as well, which gives
$$ \mu_2(\cdot, \tau) \prec \mu(\cdot, \tau) \prec \mu_1(\cdot, \tau) \text{ for all }\tau\geq 0,$$
or in other words,
$$ M(\cdot, \tau; \mu_2) \leq M(\cdot, \tau; \mu) \leq M(\cdot, \tau; \mu_1) \text{ for all }\tau\geq 0.$$

It remains to show that 
$$\sup_r |M(r, \tau; \mu_i) - M(r; \mu_A)| \leq Ce^{-\tau} \text{ for }i=1,2.$$
Recall that both $\mu_i$'s are scalings of $\mu_A$ with scaling coefficient $R_i(\tau)$, hence
\begin{equation}
|M(r, \tau; \mu_i) - M(r; \mu_A)| =
|M(\frac{r}{R_i(\tau)}; \mu_A) - M(r; \mu_A)|. \label{m_diff}
\end{equation}
Since $\mu_A$ is bounded and compactly supported, it suffices to show that $R_i(\tau) \to 1$ exponentially as $r\to \infty$.  Recall that $\dot R_i = \frac{1}{d}(\frac{1}{R_i^d}-1)R_i$ with initial data $R_{0i}$  for $i=1,2$, a simple calculation reveals that $|R_i(\tau)-1| \leq C_ie^{-\tau}$, where $C_i$ depends on $R_{0i}$.  This implies that the right hand side of \eqref{m_diff} decays like $e^{-\tau}$, which completes the proof.
\end{proof}

Making use of the explicit barriers $\mu_1$ and $\mu_2$ constructed in the proof of Proposition \ref{prop:conv_rescale}, we get exponential convergence of $\mu/A$ towards the $\mu_A/A$ in the $p$-Wasserstein metric, which is defined below. Note that the Wasserstein metric is natural for this problem, since as pointed out in \cite{ags} and \cite{cmv}, the equation \eqref{pks} is a gradient flow of the 
free energy \eqref{energy} with respect to the 2-Wasserstein metric. 
\begin{definition} \label{def_wasserstein}
Let $\mu_1$ and $\mu_2$ be two (Borel) probability measure on $\mathbb{R}^d$ with finite $p$-th moment. Then the \emph{$p$-Wasserstein distance} between $\mu_1$ and $\mu_2$ is defined as
$$
W_p(\mu_1,\mu_2) := \Big(\inf_{\pi\in \mathcal{P}(\mu_1, \mu_2)} \Big\{ \int_{\mathbb{R}^d\times\mathbb{R}^d} |x-y|^p \pi(dxdy)\Big\} \Big)^\frac{1}{p},$$
where $\mathcal{P}(\mu_1, \mu_2)$ is the set of all probability measures on $\mathbb{R}^d\times \mathbb{R}^d$ with first marginal $\mu_1$ and second marginal $\mu_2$.\end{definition}
\begin{corollary}\label{wasserstein} 
Let $d\geq 3$, and $m=2-\frac{2}{d}$. Let $\mu(\lambda,\tau)$ and $\mu_A$ be as given in Proposition \ref{prop:conv_rescale}. Then for all $p>1$, we have
$$W_p(\frac{\mu(\cdot,\tau)}{A}, \frac{\mu_A}{A}) \leq C e^{- \tau},$$
where $C$ depends on $d$ and $\mu(\cdot, 0)$.
\end{corollary}

\begin{proof}
See the proof of Corollary 5.8 in \cite{ky}. 
\end{proof}

Rescaling back to the original space and time variables, we have $$u(x,t) = \frac{1}{t+1} \mu\left(\dfrac{x}{(t+1)^{1/d}}, \ln(t+1)\right).$$ Thus Corollary \ref{wasserstein} immediately yields the algebraic convergence towards the dissipating self-similar solution \eqref{dissipating_self_similar}:

\begin{corollary}
Let $u(x,t)$ be a radial solution to \eqref{pks} with mass $0<A<M_c$, where the initial data  $u(\cdot, 0) \in L^1_+(\mathbb{R}^d; (1+|x|^2) dx) \cap L^\infty(\mathbb{R}^d)$ is continuous and compactly supported. Let $u_A$ be the dissipating self-similar solution with mass $A$ defined in \eqref{dissipating_self_similar}. Then $u/A$ converges to $u_A/A$ in Wasserstein distance algebraically fast as $t\to \infty$. More precisely, 
$$W_p(\frac{u(\cdot,t)}{A}, \frac{u_A}{A}) \leq C t^{-(d-1)/d},$$
where $C$ depends on $d, A$ and $u(\cdot, 0)$.
\label{thm:conv_m<mc}
\end{corollary}

\subsection{Convergence towards self-similar solution for subcritical mass, non-radial case}

In this subsection, we consider the rescaled equation \eqref{after_scaling} with general (possibly non-radial) initial data. The key result here is that when the mass $A < M_c$ is sufficiently small, the radially symmetric stationary solution $\mu_A$ as defined in \eqref{eq:stationary_mu} is the unique compactly supported stationary solution (in rescaled variables). Then a similar argument as in Theorem \ref{thm:m=mc} shows that every solution to \eqref{after_scaling} with small mass and compactly supported initial data converges to $\mu_A$. After scaling back to the original variables, we immediately obtain the convergence towards the self-similar solution if the mass is small.

We first prove a $L^\infty$-regularization result, saying that if the initial mass is small, then the $L^\infty$ norm of solution to \eqref{after_scaling} will become small  after unit time, regardless of the $L^\infty$ norm of the initial data. We point out that a similar  $L^\infty$-regularization result is proved in \cite{ss} for the 2D case with linear diffusion, using a De Giorgi type method.

\begin{lemma}
Suppose $d \geq 3$ and $m=2-2/d$.  Let $\mu(\lambda,\tau)$ be a weak solution to \eqref{after_scaling} with mass $0<A<M_c/2$, where the initial data  $\mu_0 \in L^1_+(\mathbb{R}^d; (1+|x|^2) dx) \cap L^\infty(\mathbb{R}^d)$ is continuous. Then we have
\begin{equation}\label{eq:max_density}
\|\mu(\cdot, \tau)\|_\infty \leq K_A := CA^{2/d} \text{ for all }\tau\geq 1,
\end{equation}
where $C$ is some constant depending only on $d$. 
\label{lemma:max_density}
\end{lemma}

\begin{proof}
Similar argument as the proof of Lemma \ref{lemma:u_bounded} yields that 
\begin{equation}
\mu^*(\cdot, \tau) \prec \bar \mu(\cdot, \tau) \text{ for all }\tau\geq 0,
\label{ineq_mu_1}
\end{equation}
where $\bar \mu(\cdot, \tau)$ is the solution to \eqref{after_scaling} with initial data $\mu^*_0$. Since $\mu^*_0$ is radially symmetric and bounded above, we can find $R_{0}$ sufficiently small, such that $\mu^*_0\prec \frac{1}{R_0^d}\mu_A(\frac{\cdot}{R_0})$, where $\mu_A$ is as defined in \eqref{eq:stationary_mu}.  It then follows from Proposition \ref{comp_concentration} and  Lemma \ref{lemma:family_solution} that
\begin{equation}
\bar\mu(\cdot, \tau) \prec  \frac{1}{R(\tau)^d}\mu_A(\frac{\cdot }{R(\tau)}) \text{ for all }\tau\geq 0,
\label{ineq_mu_2}
\end{equation}
where $R(\tau)$ satisfies the ODE \eqref{def:R1} with initial data $R(0) = R_0$. Combining \eqref{ineq_mu_1} and \eqref{ineq_mu_2}, we obtain that
$$\|\mu(\cdot, \tau)\|_\infty  = \|\mu^*(\cdot, \tau)\|_\infty \leq \frac{1}{R(\tau)^d} \|\mu_A\|_\infty \text{ for all }\tau \geq 0.$$
In order to bound the right hand side of the above inequality, we first find an upper bound for $1/R(\tau)^d$.  It can be readily verified that 
$\tilde R(\tau) = \min\{\frac{1}{2}\tau^\frac{1}{d+1}, \frac{1}{2}\}$ is a subsolution to \eqref{def:R1} for any $R_0>0$, which implies that $R(\tau)\geq \tilde R(\tau) \geq \frac{1}{2}$ for all $\tau \geq 1$, thus $\frac{1}{R(\tau)^d} \leq 2^d$ for all $\tau \geq 1$.

Next we will estimate $\|\mu_A\|_\infty$. Note that $\mu_A$ is radially decreasing for any $0<A<M_c$, moreover $\|\mu_A\|_\infty = \mu_A(0)$ is increasing with respect to $A$.  Therefore we readily obtain a rough bound $\|\mu_A\|_\infty \leq C_1$ for all $0<A<M_c/2$, where $C_1 = \mu_{M_c/2}(0)$ only depends on $d$.

Note that this rough bound of $\|\mu_A\|_\infty$ gives us an upper bound for the velocity field given by the interaction term, namely 
\begin{equation}\label{vfield}
\partial_r(\mu_A*\mathcal{N}) = \frac{M(r;\mu_A)}{\sigma_d r^{d-1}} \leq  \frac{C_1 r}{d}.
\end{equation}
To refine the bound for $\|\mu_A\|_\infty$, we compare $\mu_A$ with $\tilde \mu_A$, where $\tilde \mu_A$ is the radial stationary solution to the following equation
\begin{equation}
\mu_\tau = \Delta \mu^m + \nabla \cdot \big(\mu \nabla\frac{(1+C_1)|\lambda|^2}{2d}\big).
\label{pde_temp}
\end{equation}
Making use of \eqref{vfield}, mass comparison yields that $ \mu_A \prec \tilde \mu_A$, which implies $\mu_A(0) \leq \tilde \mu_A(0)$.  On the other hand note that  \eqref{pde_temp} is a Fokker-Planck equation, whose stationary solution is given by $$\tilde \mu_A = \Big(C_A-\frac{(1+C_1)(m-1)}{2dm} |\lambda|^2\Big)^{1/(m-1)}_+,$$ where $C_A>0$ is the unique constant such that $\|\tilde \mu_A\|_1 = A$.  A simple algebraic manipulation shows that $ \tilde\mu_A(0) \leq CA^{2/d}$, where $C>0$ depends only on $d$, therefore we can conclude.
\end{proof}

The next lemma shows that if the mass is small, any solution with compactly supported initial data will eventually be confined in some small disk. 

\begin{lemma}\label{lemma:support}
Suppose $d \geq 3$ and $m=2-2/d$.  
Then for any $R_0>0$, there exists some sufficiently small $A_0>0$, such that all weak solutions to \eqref{after_scaling} with continuous and compactly supported initial data and mass $0<A<A_0$ will be eventually confined in $B(0,R_0)$.
\end{lemma}

\begin{proof} Let $\mu(\lambda, \tau)$ be a weak solution to \eqref{after_scaling} with continuous and compactly supported initial data and mass $0<A<A_0$, where $A_0$ is a small constant depending on $R_0$ and $d$ to be determined later. 

 In the proof of this lemma we take $\tau=1$ to be the starting time, in order to take advantage of the estimate \eqref{eq:max_density}.  Our goal is to show that if the support of $\mu(\cdot, 1)$ is contained in some disk $B(0,R)$ where $R>R_0-K_A$ and $K_A$ is as defined in \eqref{eq:max_density}, then there exists some time $T>1$ to be determined later, such that
\begin{equation}\label{eq:supp1}
\text{supp}~\mu(\cdot, \tau) \subset B(0,R+K_A) \text{ for all }\tau\in[1,T],
\end{equation}
moreover at time $T$ the support can be fit into some disk smaller than $B(0,R)$, namely
\begin{equation}\label{eq:supp2}
\text{supp}~\mu(\cdot, T) \subset B(0,R-K_A/2).
\end{equation}
By taking $T$ as the starting time and repeating this procedure, we know that eventually the support will be confined in $B(0,R_0)$.

In order to deal with non-radial solution, we shall construct barriers in the density sense instead of in mass sense. Although comparison principle in density sense does not directly hold for \eqref{after_scaling} due to the nonlocal term, if we treat $V(\lambda, \tau) := \mu*\mathcal{N}$ as a fixed \emph{a priori} potential, then \eqref{after_scaling} becomes
\begin{equation}
\label{pmedrift}
\mu_\tau = \Delta \mu^m + \nabla \cdot \left(\mu \nabla \big(\frac{|\lambda|^2}{2d}+V(\lambda, \tau)\big)\right),
\end{equation}
which is a porous medium equation with a drift, and the weak solutions to it enjoy the comparison principle due to \cite{bh}.  It follows from \eqref{eq:max_density} that the following estimates of $V$ holds: 
 $$ \Delta V(\lambda, \tau) \leq \sup_{\lambda, \tau} \mu \leq K_A \text{ for }\lambda\in\mathbb{R}^d, \tau\geq 1,$$ and 
 $$|\nabla V(\lambda, \tau)| \leq \sup_{\mu,\lambda}( \mu*\frac{1}{\sigma_d |\lambda|^{d-1}}) \leq C(d) A^{\tfrac{3}{d}-\tfrac{2}{d^2}} \text{ ~for }\lambda\in\mathbb{R}^d, \tau\geq 1.$$
Note that in both estimates above, the right hand side will go to zero as $A\to 0$.  We also point out that if $R \gg A^{\frac{3}{d}-\frac{2}{d^2}}$, then $\nabla V$ will be dominated by $\nabla\frac{|\lambda|^2}{2d}$ around $r=R$.

Next we will construct some explicit supersolution $\tilde \mu$ to \eqref{pmedrift}. More precisely, we hope to find a continuous radially decreasing function $\tilde \mu$ defined in $\{ r>R-K_A\} \times [1,T]$ for some $T$, such that $\tilde \mu$ satisfies the following inequality  
\begin{equation}
\label{cond:pde}
\tilde \mu_\tau \geq \partial_{rr} \tilde \mu^m + (\partial_r + \frac{d-1}{r})(\frac{ \tilde \mu r}{d} )+ \tilde \mu K_A + |\partial_r \tilde \mu| C(d) A^{\tfrac{3}{d}-\tfrac{2}{d^2}}\text{~ for all }r>R-K_A, \tau\in[1,T],
\end{equation}
while $\tilde \mu$ also satisfies the initial condition
\begin{equation}\label{cond:init}
\tilde \mu(r,0) \geq K_A \text{~ for all } R-K_A \leq r \leq R,
\end{equation}
and the boundary condition
\begin{equation}\label{cond:bdry}
\tilde \mu(r,\tau) \geq K_A \text{ ~at } r=R-K_A \text{ for all }\tau\in[1,T].
\end{equation}
The inequalities \eqref{cond:pde}--\eqref{cond:supp} guarantees that $\tilde \mu$ is a supersolution to \eqref{pmedrift}.
If $A$ is small enough such that $R>CA^{\frac{3}{d}-\frac{2}{d^2}}$ for some large constant $C$ depending on $d$, one can check that
$$\tilde\mu(\lambda,\tau) = \big[2K_A - \tau(r-(R-K_A))\big]^{1/m}_+$$
satisfies the inequalities \eqref{cond:pde}--\eqref{cond:bdry} for $1\leq \tau \leq 4$, hence comparison principle yields that $\mu \leq \tilde \mu$  in $\{r>R-K_A\}$ for all $\tau\in [1,4]$.  

The reason we choose $\tilde \mu$ as above is that its support will shrink after some time: note that its support stays in $B(0,R+K_A)$ for $\tau\in[1,4]$, and most importantly, at $\tau=4$, the support of $\tilde \mu$ can be fit into a disk smaller than $B(0,R)$, namely
\begin{equation}\label{cond:supp}
\text{ supp }\tilde \mu(\cdot, 4) \subset B(0,R-K_A/2).
\end{equation}
Since comparison property gives that $\text{supp }\mu(\cdot, \tau) \subset \text{supp }\tilde \mu(\cdot, \tau)$ for all $\tau\in[1,4]$, we immediately obtain \eqref{eq:supp1} and \eqref{eq:supp2}, which complete the proof.
\end{proof}

Making use of the above two lemmas, in the next theorem  we show that when the mass is sufficiently small, there cannot be any non-radial stationary stationary solutions.  

\begin{theorem}
Suppose $d\geq 3$ and $m=2-2/d$. Then when $0<A<M_c/2$ is sufficiently small, the compactly supported stationary solution to \eqref{after_scaling} is unique.\label{thm:unique}
\end{theorem}

\begin{proof}Due to Corollary \ref{wasserstein}, we know that for any $0<A<M_c$, there does not exist any compactly supported radial stationary solution other than $\mu_A$.   Hence it suffices to prove that when $A$ is sufficiently small, every compactly supported stationary solution is radially symmetric.

Suppose $\nu_A(\lambda)$ is a compactly supported stationary solution to \eqref{after_scaling}, which is not radially symmetric. Since $\nu_A$ is stationary, it satisfies
\begin{equation}\label{eq:stat}
\frac{m}{m-1}\nu_A^{m-1}+ \nu_A*\mathcal{N} + \frac{|\lambda|^2}{d} = C \text{ in } \overline{\{\nu_A>0\}},
\end{equation}
where different positive components of $\nu_A$ may have different $C$'s. Heuristically, the idea is to  argue that the term $\nu_A*\mathcal{N}$ must be more ``roundish'' than $\frac{m}{m-1}\nu_A^{m-1}$ if $\nu_A$ is non-radial, thus get a contradiction.

We point out that \eqref{eq:stat} implies that $\nu_A$ is continuous in $\mathbb{R}^d$ and smooth inside its positive set. This enables us to find two points $a,b\in \mathbb{R}^d$ in the same connected component of $\overline{\{\nu_A>0\}}$,  satisfying $|a| = |b|$ and 
\begin{equation}
\nu_A(a)-\nu_A(b) = \sup_{|x|=|y|}( \nu_A(x) - \nu_A(y)) > 0.
\end{equation}
We claim that when $A$ is sufficiently small, the following inequality holds
\begin{equation}\label{ineq:diff}
\frac{m}{m-1} \big|\nu_A^{m-1}(a) - \nu_A^{m-1}(b)\big| > |(\nu_A*\mathcal{N})(a)-(\nu_A*\mathcal{N})(b)|,\end{equation}
then \eqref{ineq:diff} would contradict \eqref{eq:stat}. 

We start with the left hand side of \eqref{ineq:diff}: Lemma \ref{lemma:max_density} implies that both $\nu_A(a)$ and $\nu_A(b)$ are much smaller than $1$ when $A$ is small.  Since $0<m-1<1$, it follows that
$$\frac{m}{m-1} \big|\nu_A^{m-1}(a) - \nu_A^{m-1}(b)\big| > |\nu_A(a) - \nu_A(b)|.$$
if $A$ is sufficiently small.  In order to prove \eqref{ineq:diff}, it suffices to show that 
\begin{equation}\label{ineq:wts2}
|(\nu_A*\mathcal{N})(a)-(\nu_A*\mathcal{N})(b)| < |\nu_A(a) - \nu_A(b)|.
\end{equation}

We introduce a linear transformation $T: \mathbb{R}^d \to \mathbb{R}^d$ which is a rotation that maps $a$ to $b$. Then radial symmetry of $\mathcal{N}$ yields that $(\nu_A*\mathcal{N})(b) = ((\nu_A\circ T)*\mathcal{N})(a)$.

In addition, $T$ being a rotation implies that $|T(x)| = |x|$ for any $x=\mathbb{R}^d$, hence from the way we choose $a$ and $b$, we have $|\nu_A(T(x))-\nu_A(x)| \leq \nu_A(a)-\nu_A(b)$ for any $x\in \mathbb{R}^d$. Thus
\begin{eqnarray*}
|(\nu_A*\mathcal{N})(a)-(\nu_A*\mathcal{N})(b)| &=& \big|(\nu_A*\mathcal{N})(a)-((\nu_A\circ T)*\mathcal{N})(a)\big|  \\
&\leq & \int_{\mathbb{R}^d} \big|\nu_A(y)- \nu_A(T(y))\big| |\mathcal{N}(a-y)| dy\\
&\leq &  (\nu_A(a)-\nu_A(b)) \int_{B(0,R)} |\mathcal{N}(y)| dy,
\end{eqnarray*}
where $B(0,R)$ is the smallest disk that contains the support of $\nu_A$. Now we make use of Lemma \ref{lemma:support}, which shows that we can fit the support of $\nu_A$ into an arbitrarily small disk by letting $A$ be sufficiently small.  Therefore we can choose $R$ such that $\int_{B(0,R)} |\mathcal{N}(y)| dy < 1/2$, then let $A$ be sufficiently small such that $\text{supp}~\nu_A \subset B(0,R)$. This gives us \eqref{ineq:wts2}, which leads to a contradiction and hence completes the proof.
\end{proof}

\begin{remark}\textup{
For general $0<A<M_c$, we are unable to prove the uniqueness of the compactly supported stationary solution.  The difficulty lies in the fact that for larger mass we are only able to show the support lies in a disk with radius $O(1)$. Hence instead of \eqref{ineq:wts2}, we can only obtain $|(\nu_A*\mathcal{N})(a)-(\nu_A*\mathcal{N})(b)| < C|\nu_A(a) - \nu_A(b)|,$ where $C$ might be a large constant, which stops us from getting a contradiction.
}
\end{remark}

Once we obtain the uniqueness of compactly supported stationary solution for small mass, the following corollary shows that all solution with compactly supported initial data must converge to this unique stationary solution as $\tau\to\infty$.

\begin{corollary}\label{cor:nonradial}
Suppose $d \geq 3$ and $m=2-2/d$.  Let $\mu(\lambda,\tau)$ be a weak solution to \eqref{after_scaling} with mass $0<A<M_c/2$ being sufficiently small, where the initial data  $\mu(\cdot, 0) \in L^1_+(\mathbb{R}^d; (1+|x|^2) dx) \cap L^\infty(\mathbb{R}^d)$ is continuous and compactly supported. Then as $\tau\to\infty$, we have
\begin{equation}\label{eq:unif_conv}
\|\mu(\cdot, \tau) - \mu_A(\cdot)\|_\infty \to 0,
\end{equation}
where $\mu_A$ is as defined in \eqref{eq:stationary_mu}.
\end{corollary}

\begin{proof}
The proof is similar as the proof of Theorem \ref{thm:m=mc}, and actually it is simpler here since there is a unique stationary solution, instead of a family of stationary solution in the case of Theorem \ref{thm:m=mc}. 

When the initial data $\mu(\cdot, 0)$ is bounded and compactly supported, Lemma \ref{lemma:max_density} and Lemma \ref{lemma:support} shows that $\mu(\cdot, \tau)$ would be uniformly bounded and stay in some fixed compact set for all $\tau\geq1$. In addition, the continuity result in \cite{dib} indicates that $\mu(\lambda,\tau)$ is uniformly continuous in space and time in $\mathbb{R}^d\times [1, \infty)$.

As a result, for any time sequence $\tau_n$ that increases to infinity, using the same argument as in the proof of Theorem \ref{thm:m=mc},  we can extract a subsequence $\tau_{n_k}$ such that $\mu(\cdot, \tau_{n_k})$ uniformly converges to some continuous function $\mu_\infty$, where $\mu_\infty$ is a compactly supported stationary solution. Theorem \ref{thm:unique}  ensures that $\mu_\infty$ must coincide with $\mu_A$ when $A$ is sufficiently small, yielding that $\mu(\cdot, \tau)$ indeed converges to $\mu_A$ uniformly as $\tau\to \infty$.
\end{proof}

\begin{remark}
Since $\mu(\cdot, 0)$ is confined in some compact set for all time,  \eqref{eq:unif_conv} implies that $\|\mu(\cdot, \tau) - \mu_A(\cdot)\|_p \to 0$ as $\tau\to\infty$ for all $p\geq 1$.  Now if we scale back to the original variables, it immediately follows that $\|u(\cdot, t) - u_A(\cdot)\|_p \to 0$ as $t\to\infty$ for all $p\geq 1$, where $u_A$ is the dissipating self-similar solution as defined in \eqref{dissipating_self_similar}.  However the rate of convergence here is unknown, since the proof  is done by extracting a subsequence of time.
\end{remark}

\section{Application to aggregation models with repulsive-attractive interactions}

In this section we consider the following integro-differential equation
\begin{equation}
u_t = \nabla \cdot (u \nabla K* u) , \label{rep_att}
\end{equation}
where the interaction kernel $K$ has a repulsion component in the form of the Newtonian potential $\mathcal{N}(x) = -\frac{1}{(d-2) \sigma_d |x|^{d-2}}$ and an attraction component satisfying the power law, namely
\begin{equation}
K(x) = \mathcal{N}(x) + \frac{1}{q}|x|^q, \label{def:K}
\end{equation}
where $2-d<q\leq 2$, and when $q=0$ the second term is replaced by $\ln|x|$. 

The global existence of weak solution is established in \cite{fh} for $q>2-d$. Next we show that mass comparison holds for \eqref{rep_att} between weak solutions. 

\begin{proposition}
Let $u_1(x,t)$, $u_2(x,t)$ be two radially symmetric weak solution to \eqref{rep_att}, which are compactly supported for all $t\geq 0$.  If $u_1(\cdot, 0) \prec u_2(\cdot, 0)$, then $u_1(\cdot, t) \prec u_2(\cdot, t)$ for all $t\geq 0$.
\end{proposition}

\begin{proof}
Note that $\Delta \frac{1}{q}|x|^q = (q+d-2)|x|^{q-2}$, which is locally integrable in $\mathbb{R}^d$, nonnegative and radially decreasing when $2-d < q \leq 2$.   Therefore the conditions \textbf{(C), (K1), (K2'), (V1)} are met, (where \textbf{(K2')} is as defined in  Remark \ref{remark:l1loc}), and \eqref{rep_att} becomes a special case of the general equation \eqref{general_pde} in Section 2.  Due to the discussion in Remark \ref{remark:l1loc}, we can apply mass comparison to the compactly supported solutions.  

Without loss of generality we assume that both $u_1(\cdot, 0)$ and $u_2(\cdot,0)$ are continuous, and for general initial data we can use approximation. Due to Theorem 2.5 of \cite{fh}, we have $M_i$ is $C^1$ in both space and time for all $t$, where $i=1,2$.  Now we can apply  Proposition \ref{comp_concentration} to conclude that $u_1(\cdot, t) \prec u_2(\cdot, t)$ for all time: although the proof of Proposition \ref{comp_concentration} requires $M_i$ be $C^2$ in space, the $C^2$ requirement are only for the diffusion term.  Since the right hand side of \eqref{rep_att} only has aggregation terms, $C^1$ continuity of $M_i$ is sufficient.
\end{proof}

The following existence and uniqueness result of a stationary solution is established in \cite{fh}. 
\begin{proposition}[\cite{fh}, Theorem 3.1]
For every $q>2-d$ and mass $A>0$, there exists a unique radius $R_A$ (that depends on $q$ and $d$ only) and a unique steady state $u_s$ of the aggregation model \eqref{rep_att}-\eqref{def:K} that is supported on $B(0, R_A)$, has mass $A$ and is continuous on its support.\label{prop:stat_sol}

In addition, $u_s$ is radially decreasing in $B(0,R_A)$ if $q\leq 2$, and is radially increasing if $q\geq 2$.
\end{proposition}

It is proved by \cite{fh} that in its positive set, $u_s$ satisfies $\nabla K* u_s = 0$, which can also be written as $-u_s + \Delta (\frac{1}{q}|x|^q)*u_s = 0$, i.e.
\begin{equation}\label{eq:u_pde}
-u_s + (q+d-2)|x|^{q-2} * u_s = 0 \quad\text{in }\{u_s>0\}.
\end{equation}

In the proposition below, we construct a family of explicit subsolutions, all of which are compactly supported and converge to the stationary solution $u_s$ exponentially fast.

\begin{lemma}[\textbf{A family of explicit subsolutions}] Suppose $d\geq 3$ and $2-d < q \leq 2$. Let $u_s$ be the stationary solution to \eqref{rep_att} with mass $A$, as given by Proposition \ref{prop:stat_sol}.  We define the self-similar function $\bar u$ as 
\begin{equation}
\bar u (x,t) := \dfrac{1}{R^{d}(t)} u_s(\dfrac{x }{R(t)}),\label{def:u_bar}
\end{equation} 
where  $R(t)$ solves the ODE
\begin{equation}\left\{
\begin{split}
\dot{R}(t) &= C_1(1-R^{d+q-2}) R^{-d+1}\\
R(0) &= R_0,
\end{split}\right. \label{def:R}
\end{equation}
where $R_0 > 1$, and $C_1$ is some fixed constant only depending on $q, d$ and $A$.  Then for all $t\in [0,\infty)$, $\bar u(x,t)$ is a subsolution to \eqref{rep_att} in the mass comparison sense.
\label{lemma:subsol}
\end{lemma}

\begin{proof}
The proof is similar to the proof of Lemma \ref{lemma:family_solution}. Since $\bar u$ is a self-similar function, it can be easily verified that $\bar u $ solves the following transport equation
$$\bar u_t + \nabla \cdot (\bar u \frac{\dot R(t)}{R(t)} x)=0,$$
which implies that the mass function of $\bar u$ satisfies
\begin{equation}
M_t(r,t;\bar u) = \sigma_d r^{d-1} M_r(r,t;\bar u) ~(-\frac{\dot R(t)}{R(t)}) r. \label{m_eq_1}
\end{equation}
Our goal is to show that $\bar u$ is a subsolution to \eqref{rep_att} in mass comparison sense, which is equivalent to the following inequality due to Lemma \ref{pde_for_m}:
\begin{equation}
M_t(r,t;\bar u) \leq  \sigma_d r^{d-1} M_r(r,t;\bar u) \left[  \frac{M(r,t;\bar u)}{\sigma_d r^{d-1} } + \frac{\tilde M(r,t;\bar u)}{\sigma_d r^{d-1} }\right ].\label{m_eq_2}
\end{equation}
where $\tilde M(r,t; \bar u) = \int_{B(0,r)} (q+d-2)(|x|^{q-2} * \bar u)(x,t)dx$.  

We point out that $M_r$ is nonnegative by definition, since $M_r(r,t;\bar u) = \sigma_d r^d \bar u$, where $\bar u$ is nonnegative. By comparing \eqref{m_eq_1} and \eqref{m_eq_2}, it suffices to prove that the following inequality holds for $\bar u(r,t)$ for all $0\leq r\leq R(t)$:
\begin{equation}
-\frac{\dot R(t)}{R(t)} r \leq -\frac{M(r,t;\bar u)}{\sigma_d r^{d-1}} + \frac{\tilde M(r,t; \bar u)}{\sigma_d r^{d-1}}, \label{ineq:wts}
\end{equation}
Next we will investigate the terms on the right hand side of \eqref{ineq:wts}. Since $\bar u$ is defined as a continuous scaling of $u_s$, for the first term on the right hand side of \eqref{ineq:wts}, we have
$$M(r,t; \bar u) = M(\frac{r}{R(t)}; u_s).$$
For the second term, we obtain that for any $0\leq r\leq R(t)$, 
\begin{eqnarray*}
\tilde M(r,t; \bar u) &=& (q+d-2) \int_{B(0,r)} \int_{\mathbb{R}^d} |x-y|^{q-2} \frac{1}{R(t)^d} u_s(\frac{y}{R(t)}) dy dx\\
&=&(q+d-2) R(t)^{q-2} \int_{B(0,r)} \int_{\mathbb{R}^d} |\frac{x}{R(t)}-z|^{q-2} u_s(z) dz dx ~~ (z:=y/R(t))\\
&=& (q+d-2) R(t)^{q-2} \int_{B(0,r)} (|x|^{q-2} * u_s)(\frac{x}{R(t)}) dx\\
&=& R(t)^{q-2} \int_{B(0,r)} u_s(\frac{x}{R(t)}) dx \text{ ~(by \eqref{eq:u_pde})}\\
&=& R(t)^{d+q-2}M(\frac{r}{R(t)}; u_s).
\end{eqnarray*}
Putting the above two equations together yields
\begin{eqnarray}
\nonumber \text{RHS of \eqref{ineq:wts}} &=& (-1+R(t)^{d+q-2}) \frac{M(\frac{r}{R(t)}; u_s)}{\sigma_d r^{d-1}}\\
&\geq & C_1(-1+R(t)^{d+q-2})R(t)^{-d}r. \label{ineq:m_tilde}
\end{eqnarray}
Where $C_1$ only depends on $q, d$ and $A$. Here in the last inequality we used the fact that $u_s$ is radially decreasing, which implies that $M(r; u_s) \geq C_A |B(0,r)|$ for all  $0\leq r \leq R_A$, where $C_A$ is the average density of $u_s$ in its support $B(0,R_A)$.

Since $\dot R = C_1(1-R^{d+q-2}) R^{-d+1}$ by definition,  the above inequality implies that \eqref{ineq:wts} is true, which completes the proof.
\end{proof}
\begin{remark}
Similarly, we can construct a family of explicit supersolutions $\bar u$ in the mass comparison sense: here $\bar u$ is defined in \eqref{def:u_bar}, where $R(t)$ solves the ODE \eqref{def:R} with initial data $0<R_0<1$, and the constant $C_1$ in \eqref{def:R} is replaced by some other fixed constant $C_2$, which also only depends on $q, d$ and $A$.  
\label{rmk:supersol}
\end{remark}

Making use of the subsolutions and supersolutions we constructed in Lemma \ref{lemma:subsol} and Remark \ref{rmk:supersol} respectively, we next prove that all radial solutions with compactly supported initial data will converge to the unique stationary solution exponentially fast. 

\begin{theorem}\label{thm:aggregation}
Suppose $d\geq 3$ and $2-d < q \leq 2$.  Let $u$ be a weak solution to \eqref{rep_att} with initial data $u_0$ and mass $A$, where $u_0 \in L^1(\mathbb{R}^d) \cap L^\infty(\mathbb{R}^d) $ is non-negative, radially symmetric and compactly supported.  In addition, we assume that $u_0$ is strictly positive in a neighborhood of $0$. Let $u_s$ be the unique stationary solution with mass $A$, as given by Proposition \ref{prop:stat_sol}. Then as $t\to \infty$, $u(\cdot, t)$ converges to $u_s$ exponentially fast in Wasserstein distance.
\end{theorem}

\begin{proof}
The proof is similar to the proof of Proposition \ref{prop:conv_rescale}.  Since $u_0$ is bounded, and is strictly positive in a neighborhood of 0, we can find $R_{01}$ sufficiently large and $R_{02}$ sufficiently small, such that
$$\dfrac{1}{R_{01}^d}u_s(\dfrac{\cdot}{R_{01}}, 0)  \prec u_0 \prec \dfrac{1}{R_{02}^d}u_s(\dfrac{\cdot}{R_{02}}, 0).$$
Let $\bar u_1$ be the subsolution given by Lemma \ref{lemma:subsol} with initial data $\frac{1}{R_{01}^d}u_s(\frac{\cdot}{R_{01}}, 0)$, and $\bar u_2$ be the supersolution given by Remark \ref{rmk:supersol} with initial data $\frac{1}{R_{02}^d}u_s(\frac{\cdot}{R_{02}}, 0)$.  Then mass comparison  in Proposition \ref{comp_concentration}  yields $\bar u_1(\cdot, t) \prec u(\cdot, t) \prec \bar u_2(\cdot, t)$ for all $t$. 

Now it suffices to show that $\bar u_1$ and $\bar u_2$ both converges to $u_s$ exponentially fast in Wasserstein distance. Note that $\bar u_1(\cdot, t)$ is a continuous scaling of $u_s$ with scaling coefficient $R(t)$, where $R(t)$ satisfies the ODE \eqref{def:R} and hence converges to 1 exponentially. More precisely, we have 
$$|R(t)-1| \leq C_1' e^{-C_1 (d+q-2) t},$$
where $C_1$ is as given in Lemma \ref{lemma:subsol},  and $C_1'$ depends on $q, d$ and $R_{01}$.  Similar result hold for the supersolution $\bar u_2$. Then argue as in Corollary \ref{wasserstein}, we have for all $p>1$ that 
$$W_p(\frac{u_i(\cdot,t)}{A}, \frac{u_s}{A}) \leq c_1 e^{-C_i(d+q-2) t} \quad \text{for }i = 1,2,$$
where $c_i$ depends on $R_{0i}$ respectively. Since $u_s$ is squeezed between $\bar u_1$ and $\bar u_2$ in the mass comparison sense, the inequality above yields
$$W_p(\frac{u(\cdot,t)}{A}, \frac{u_s}{A}) \leq C' e^{-C(d+q-2) t} \quad \text{for }i = 1,2,$$
where $C := \min(C_1, C_2)$ depends on $d, q$ and $A$, while $C'$ depends on $d, q, A$ and $u(\cdot, 0)$.
\end{proof}

\end{document}